\documentclass{amsart}
\usepackage{enumerate,url,subfigure,tikz,hyperref}

\newtheorem{theorem}{Theorem}
\newtheorem{proposition}[theorem]{Proposition}
\newtheorem{lemma}[theorem]{Lemma}
\theoremstyle{remark}
\newtheorem{remark}[theorem]{Remark}

\title{Permutations and negative beta-shifts}

\author{\'Emilie Charlier}
\address{Institute of Mathematics, University of Li\`ege, All\'ee de la d\'ecouverte 12 (B37) \\ 4000 Li\`ege, Belgium\\ echarlier@ulg.ac.be}
\author{Wolfgang Steiner}
\address{IRIF, CNRS UMR 8243, Universit\'e Paris Diderot -- Paris 7, Case 7014 \\ 75205 Paris Cedex 13, France \\ steiner@irif.fr}

\begin{document}
\begin{abstract}
Elizalde (2011) characterized which permutations can be obtained by ordering consecutive elements in the trajectories of (positive) beta-transformations and beta-shifts. We prove similar results for negative bases beta.
\end{abstract}

\maketitle
\section{Introduction}
The complexity of a dynamical system is usually measured by its entropy. 
For symbolic dynamical systems, the (topological) entropy is the logarithm of the exponential growth rate of the number of distinct patterns of length~$n$. 
Bandt, Keller and Pompe~\cite{Bandt-Keller-Pompe02} proved for piecewise monotonic maps that the entropy is also given by the number of permutatitions defined by consecutive elements in the trajectory of a point. 
Amig\'o, Elizalde and Kennel~\cite{Amigo-Elizalde-Kennel08} and Elizalde~\cite{Elizalde09} studied realizable permutations in full shifts in detail. 
Elizalde~\cite{Elizalde11} extended this study to $\beta$-shifts (with $\beta > 1$), and he determined for each permutation the infimum of those bases~$\beta$ where successive elements of the $\beta$-shift are ordered according to the permutation. 
Archer and Elizalde~\cite{Archer-Elizalde14} considered periodic patterns for full shifts with different orderings. 

We are interested in $\beta$-shifts with $\beta < -1$, which are ordered naturally by the alternating lexicographical order.
While several properties for positive bases have analogs for negative bases, negative $\beta$-shifts also exhibit interesting new phenomena. 
For example, Liao and Steiner~\cite{Liao-Steiner12} showed that the support of the unique absolutely continuous invariant measure of the $\beta$-transformation has more and more gaps as $\beta \to -1$, and they determined the combinatorial structure of the gaps. 

The $\beta$-shift is governed by the $\beta$-expansion of~$1$.
This expansion becomes trivial as $\beta \to 1$, while it is the fixed point of a primitive substitution for $\beta \to -1$, which is aperiodic. 
Hence the only permutations that occur in all $\beta$-shifts with $\beta > 1$ are of the form $j (j{+}1) \cdots n 1 2 \cdots (j{-}1)$, while more permutations are possible for $\beta < -1$. 
Similarly to \cite{Elizalde11}, we determine the set of $(-\beta)$-shifts allowing a given permutation. 
Our main result (Theorem~\ref{t:main}) was obtained independently by Elizalde and Moore~\cite{Elizalde-Moore}.

\section{Definitions and main results} \label{sec:defin-main-results}
For an ordered space~$X$, a map $f:\, X \to X$, a positive integer~$n$, a point $x \in X$ such that $f^i(x) \ne f^j(x)$ for all $0 \le i < j < n$, and a permutation $\pi \in \mathcal{S}_n$, let
\[
\mathrm{Pat}(x, f, n) = \pi \quad \mbox{if} \quad \pi(i) < \pi(j)\ \mbox{for all}\ 1 \le i,  j \le n\ \mbox{with}\ f^{i-1}(x) < f^{j-1}(x).
\]
Otherwise stated, for all $1 \le i \le n$, $\pi(i)=j$ if $f^{i-1}(x)$ is the $j$th element in the ordered list $x,f(x),\ldots,f^{n-1}(x)$. 
For example, if $n=3$ and $f^2(x)<x<f(x)$, then $\mathrm{Pat}(x, f, 3) =231$.
The set of allowed patterns of~$f$ is
\[
\mathcal{A}(f) = \big\{\mathrm{Pat}(x, f, n):\, x \in X,\, n \ge 1\big\}.
\]

For $\beta > 1$, the $\beta$-transformation is 
\[
T_\beta:\, [0,1) \to [0,1), \quad x \mapsto \beta x - \lfloor \beta x\rfloor,
\] 
and Elizalde~\cite{Elizalde11} gave a formula for 
\[
B_+(\pi) = \inf\big\{\beta > 1:\, \pi \in \mathcal{A}(T_{\beta})\big\}.
\] 
Here, we are interested in the \emph{$(-\beta)$-transformation}, which was defined by Ito and Sadahiro~\cite{Ito-Sadahiro09} as $x \mapsto \lfloor \frac{\beta}{\beta+1} - \beta x\rfloor -\beta x$ on the interval $[\frac{-\beta}{\beta+1}, \frac{1}{\beta+1})$. 
We find it more convenient to use the map
\[
T_{-\beta}:\, (0,1] \to (0,1], \quad x \mapsto \lfloor \beta x\rfloor + 1 -\beta x,
\]
which is easily seen to be topologically conjugate to Ito and Sadahiro's one, via $x \mapsto \frac{1}{\beta+1} -x$ (which reverses the order between elements); see Figure~\ref{f:T}.
Theorem~\ref{t:main} below gives a formula for
\[
B_-(\pi) = \inf\big\{\beta > 1:\, \pi \in \mathcal{A}(T_{-\beta})\big\}.
\] 

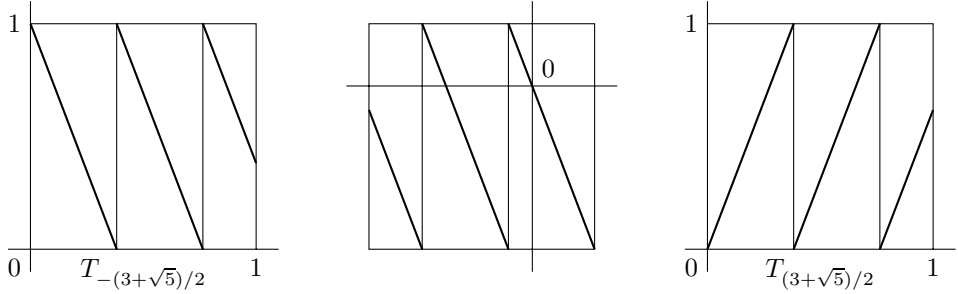
\begin{figure}[ht]
\centering
\begin{tikzpicture}[scale=3]
\draw(-.1,0)--(1.1,0) (0,-.1)--(0,1.1) (1,0)--(1,1)--(0,1) (.382,0)--(.382,1) (.764,0)--(.764,1);
\node[below left] at (0,0){$0$};
\node[below] at (1,0){$1$};
\node[left] at (0,1){$1$};
\node[below] at (.5,0){$T_{-(3+\sqrt{5})/2}$};
\draw[thick](0,1)--(.382,0) (.382,1)--(.764,0) (.764,1)--(1,.382);

\begin{scope}[shift={(2.224,.724)}]
\draw(-.824,0)--(.376,0) (0,-.824)--(0,.376) (-.724,-.724)--(.276,-.724)--(.276,.276)--(-.724,.276)--cycle (-.106,-.724)--(-.106,.276) (-.488,-.724)--(-.488,.276);
\node[above right] at (0,0){$0$};
\draw[thick](-.724,-.106)--(-.488,-.724) (-.488,.276)--(-.106,-.724) (-.106,.276)--(.276,-.724);
\end{scope}

\begin{scope}[shift={(3,0)}]
\draw(-.1,0)--(1.1,0) (0,-.1)--(0,1.1) (1,0)--(1,1)--(0,1) (.382,0)--(.382,1) (.764,0)--(.764,1);
\node[below left] at (0,0){$0$};
\node[below] at (1,0){$1$};
\node[left] at (0,1){$1$};
\node[below] at (.5,0){$T_{(3+\sqrt{5})/2}$};
\draw[thick](0,0)--(.382,1) (.382,0)--(.764,1) (.764,0)--(1,.618);
\end{scope}
\end{tikzpicture}
\caption{$(-\beta)$-transformation $T_{-\beta}$ (left), Ito and Sadahiro's version (middle) and $\beta$-transformation (right), $\beta \approx 2.618$.}
\label{f:T}
\end{figure}

Instead of numbers $x \in (0,1]$, we will rather consider their \emph{$(-\beta)$-expansions} 
\[
x = - \sum_{k=1}^\infty \frac{d_{-\beta,k}(x)+1}{(-\beta)^k} \quad
\mbox{with} \quad d_{-\beta,k}(x) = \big\lfloor \beta\, T_{-\beta}^{k-1}(x)\big\rfloor.
\]
Set $d_{-\beta}(x) = d_{-\beta,1}(x) d_{-\beta,2}(x) \cdots$, and note that $d_{-\beta}(x)\in\{0,1,\ldots,\lfloor\beta\rfloor\}^\infty$.
Then we have $x < y$ if and only if $d_{-\beta}(x) < d_{-\beta}(y)$.
Here and throughout the paper, we use the \emph{alternating lexicographical order} for sequences (or infinite words):
\[
v_1 v_2 \cdots < w_1 w_2 \cdots \quad \mbox{if}\ v_1 \cdots v_{k-1} = w_1 \cdots w_{k-1}\ \mbox{and}\ \begin{cases}v_k < w_k & \mbox{when $k$ is odd}, \\ w_k < v_k & \mbox{when $k$ is even},\end{cases}
\]
for some $k \ge 1$; we will also use it to compare finite words of same length. 
Note that Ito and Sadahiro used an ``alternate order'', which is the inverse of our alternating lexicographical order.
The set of $(-\beta)$-expansions forms the \emph{$(-\beta)$-shift} 
\[
\Omega_{-\beta} = \{d_{-\beta}(x):\, x \in (0,1]\},
\]
which is stable under the shift map $\Sigma\colon w_1 w_2 \cdots \mapsto w_2 w_3 \cdots$. 
Thus, for all $x\in(0,1]$,
\[
\mathrm{Pat}(x,T_{-\beta},n) = \mathrm{Pat}(d_{-\beta}(x),\Sigma,n),
\]
with the alternating lexicographical order on the $(-\beta)$-shift.

By Theorem~10 of~\cite{Ito-Sadahiro09}, we have $w_1 w_2 \cdots \in \Omega_{-\beta}$ if and only if, for all $k \ge 1$, 
\[
d_{-\beta}(1) \ge w_{[k,\infty)} > 
	\begin{cases}\overline{0 d_{-\beta,1}(1) \cdots d_{-\beta,p-1}(1) (d_{-\beta,p}(1){-}1)} 
					&	\mbox{if $d_{-\beta}(1)$ is purely} \\ 
					& 	\hspace{-11em} \mbox{periodic with odd minimal period length $p$}, \\ 
			0 d_{-\beta}(1) & \mbox{otherwise}.
	\end{cases}
\]
We denote by $\overline{w_1 w_2 \cdots w_n}$ the purely periodic sequence with period $w_1 w_2 \cdots w_n$, 
\[
w_{[k,\infty)} = w_k w_{k+1} \cdots, \quad w_{[i,j)} = w_i w_{i+1} \cdots w_{j-1}, \quad w_{[i,j]} = w_i w_{i+1} \cdots w_j..
\]
Note that the ordered space $\Omega_{-\beta}$ is not closed; taking the closure of $\Omega_{-\beta}$ amounts to replacing the strict lower bound for $w_{[k,\infty)}$ by a non-strict one. 
By Theorem~3 of~\cite{Steiner13}, we know that whenever $1 < \alpha < \beta$ we have $d_{-\alpha}(1) < d_{-\beta}(1)$, and hence $\Omega_{-\alpha} \subseteq \Omega_{-\beta}$ and $\mathcal{A}(T_{-\alpha}) \subseteq\mathcal{A}(T_{-\beta})$.

A~number $\beta > 1$ is an \emph{Yrrap number} if $d_{-\beta}(1)$ is eventually periodic. 
By Theorem~2.6 in~\cite{Liao-Steiner12}, each Yrrap number is a Perron number, i.e., an algebraic integer $\beta > 1$ with all its Galois conjugates (except itself) less than $\beta$ in absolute value.
On the other hand, not every Perron number is an Yrrap number. 
However, each Pisot number is an Yrrap number by Theorem~5.15 in~\cite{Frougny-Lai11}.

For a bounded sequence $w = w_1 w_2 \cdots \in \mathbb{N}^\infty$, where $\mathbb{N}$ denotes the set of nonnegative integers, let
\[
\widehat{w} = \sup_{k\ge1} w_{[k,\infty)}.
\]
Let $\varphi$ be the substitution defined by $\varphi(0) = 1$, $\varphi(1) = 100$, with the unique fixed point $u = \varphi(u)$, i.e.,
\[
u = 1 0 0 1 1 1 0 0 1 0 0 1 0 0 1 1 1 0 0 1 1 \cdots.
\]
If $\widehat{w} = w$ and $w > u$, then let $b(w) > 1$ be the largest positive solution $x$ of 
\[
1 + \sum_{k=1}^\infty \frac{w_k+1}{(-x)^k} = 0,
\]
which exists by Corollary~1 of~\cite{Steiner13}.
Note that $b(w) \le \max_{k\ge1} w_k+1$.
If $w$ is eventually periodic with preperiod of length~$q$ and period of length~$p$, then $b(w)$ is the largest positive root of the polynomial
\[
\bigg((-x)^{p+q} + \sum_{k=1}^{p+q} (w_k+1)\, (-x)^{p+q-k}\bigg) - \bigg((-x)^q + \sum_{k=1}^q (w_k+1)\, (-x)^{q-k}\bigg). 
\]
If $\widehat{w} = w$ and $w \le u$, then we set $b(w) = 1$. 

Throughout the paper, let $\pi \in \mathcal{S}_n$ be an arbitrary but fixed permutation and
\[
m = \pi^{-1}(n), \quad \ell = \pi^{-1}(\pi(n)-1) \ \mbox{if}\ \pi(n) \ne 1, \quad r = \pi^{-1}(\pi(n)+1) \ \mbox{if}\ \pi(n) \ne n.
\]
The sequence of digits $z_{[1,n)}$ defined by
\begin{align*}
z_j = \# \{1 \le i < \pi(j):\ & i \ne \pi(n) \ne i+1,\ \pi(\pi^{-1}(i)+1) < \pi(\pi^{-1}(i+1)+1), \\ & \mbox{or}\ i+1 = \pi(n) \ne n,\ \pi(\ell+1) < \pi(r+1)\}
\end{align*}
will play an important role. 
See the examples in Section~\ref{sec:examples} for an effective computation of these digits~$z_j$.
Following~\cite{Elizalde-Moore}, we say that $\pi$ is \emph{collapsed} if $\pi(n) \notin \{1,n\}$ and $z_{[\ell,n)} = z_{[r,n)} z_{[r,n)}$ or $z_{[r,n)} = z_{[\ell,n)} z_{[\ell,n)}$.
In this case, we also use the digits
\[
z_j^{(i)} = z_j + \begin{cases}1 & \mbox{if}\ \pi(j) \ge \pi(r+i)\ \mbox{and $i$ is even},\ \mbox{or}\ \pi(j) \ge \pi(\ell+i)\ \mbox{and $i$ is odd}, \\ 0 & \mbox{otherwise},\end{cases}
\]
for $0 \le i < |r-\ell|$, $1 \le j < n$.

\begin{theorem} \label{t:main}
Let $\pi \in \mathcal{S}_n$, $\beta > 1$.
We have $\pi \in \mathcal{A}(T_{-\beta})$ if and only if $\beta > b(a)$, 
with
\[
a = \begin{cases}z_{[m,n)}\, \overline{z_{[\ell,n)}} & \mbox{if $n-m$ is even, $\pi(n) \ne 1$, and $\pi$ is not collapsed}, \\ 
\overline{z_{[m,n)} 0} & \mbox{if $n-m$ is even and}\ \pi(n) = 1, \\[.5ex] 
\min_{0\le i<|r-\ell|} z^{(i)}_{[m,n)}\, \overline{z^{(i)}_{[\ell,n)}} & \mbox{if $n-m$ is even and $\pi$ is collapsed}, \\[.5ex] 
z_{[m,n)}\, \overline{z_{[r,n)}} & \mbox{if $n-m$ is odd and $\pi$ is not collapsed}, \\
\min_{0\le i<|r-\ell|} z^{(i)}_{[m,n)}\, \overline{z^{(i)}_{[r,n)}} & \mbox{if $n-m$ is odd and $\pi$ is collapsed}.\end{cases}
\]
In particular $B_-(\pi) = b(a)$, and $B_-(\pi)$ is either~$1$ or an Yrrap number. 
\end{theorem}

Note that $z_{[m,n)}\, \overline{z_{[\ell,n)}} = \overline{z_{[m,n)} z_{[\ell,m)}}$ if $\ell < m$, $z_{[m,n)}\, \overline{z_{[\ell,n)}} = z_{[m,\ell)}\, \overline{z_{[\ell,n)}}$ if $\ell > m$.

\begin{theorem} \label{t:1}
We have $B_-(\pi) = 1$ if and only if $a = \overline{\varphi^k(0)}$ for some $k \ge 0$. 
\end{theorem}

It would be interesting to count the permutations with $B_-(\pi) = 1$.
From Bandt, Keller and Pompe~\cite{Bandt-Keller-Pompe02}, we know that $\lim_{n\to\infty} \frac{1}{n} \log \#\{\pi \in \mathcal{S}_n:\, B_-(\pi) < \beta\} = \log \beta$ (which is the entropy of the $(-\beta)$-transformation) for $\beta > 1$, but we do not know whether $c_n = \#\{\pi \in \mathcal{S}_n:\, B_-(\pi) = 1\}$ grows polynomially; we have
\[
c_2 = 2,\ c_3 = 5,\ c_4 = 12,\ c_5 = 19,\ c_6 = 34,\ c_7 = 57,\ c_8 = 82,\ c_9 = 115,\ \dots
\]
Recall that $\#\{\pi \in \mathcal{S}_n:\, B_+(\pi) = 1\} = n$. 
It would also be interesting to give more precise asymptotics for the number of permutations with $B_-(\pi) < N$ or $B_-(\pi) \le N$ for some integer $N \ge 2$. 
The following theorem characterizes $B_-(\pi) < N$.

\begin{theorem} \label{t:N}
Let $\pi \in \mathcal{S}_n$, $n \ge 2$.
The minimal number of distinct elements of a sequence~$w$ satisfying $\mathrm{Pat}(w,\Sigma,n) = \pi$, w.r.t.\ the alternating lexicographical order,~is
\[
N_-(\pi) = \lfloor B_-(\pi) \rfloor + 1 = \max_{1\le j<n} z_j + 1 + \epsilon,
\]
with $\epsilon = 1$ if $\pi$ is collapsed or $a = \max_{1\le j<n} \overline{z_j 0}$, $\epsilon = 0$ otherwise. 
We have $N_-(\pi) \le n-1$ for all $\pi \in \mathcal{S}_n$, $n \ge 3$, with equality for $n \ge 4$ if and only if
\[
\pi \in \{12\cdots n,\ 12\cdots(n{-}2)n(n{-}1),\ n(n-1)\cdots1,\ n(n-1)\cdots312\}.
\]
Moreover, for $n\ge 3$, we have
\[
\max_{\pi\in\mathcal{S}_n} B_-(\pi) = b\big((n{-}2)(n{-}3)\cdots 1\overline{0}\big) \in (n-2,n-1),
\]
the maximum is attained only for $\pi = n(n{-}1)\cdots1$ if $n$ is even, $\pi = n(n{-}1)\cdots 312$ if $n$ is odd. 
\end{theorem}

We mention without proof that $b((n{-}2)(n{-}3)\cdots 1\overline{0})$, which is the largest root of $x^{n-1}-(n-2)x^{n-2}+ (-1)^n \sum_{i=0}^{n-3} (-x)^i$, is a Pisot number.

We prove that the numbers occurring as $B_-(\pi)$ are exactly the Yrrap numbers. 
The corresponding question for $B_+(\pi)$ is an open problem of Elizalde~\cite{Elizalde11}.

\begin{theorem} \label{t:Yrrap}
Let $\beta > 1$. 
We have $B_-(\pi) = \beta$ for some $\pi \in \mathcal{S}_n$, $n \ge 1$, if and only if $\beta$ is an Yrrap number. 
\end{theorem}

Following~\cite{Elizalde11}, it is sometimes convenient to use the circular permutation
\[
\tilde{\pi} = \big(\pi(1) \pi(2) \cdots \pi(n)\big) \in \mathcal{S}_n,
\]
i.e., $\tilde{\pi}(\pi(j)) = \pi(j+1)$ for $1 \le j < n$, $\tilde{\pi}(\pi(n)) = \pi(1)$.
Thanks to this notation~$\tilde{\pi}$, we get another formulation of the digits $z_j$:
\begin{align*}
z_j = \# \{1 \le i < \pi(j):\ & i \ne \pi(n) \ne i+1,\ \tilde{\pi}(i) < \tilde{\pi}(i+1), \\ 
						& \mbox{or}\ i+1 = \pi(n) \ne n,\ \tilde{\pi}(i) < \tilde{\pi}(i+2)\}.
\end{align*}
In particular, $\max_{1\le i<n} z_i$ equals the number of ascents in~$\tilde{\pi}$ with $\pi(1)$ removed.

\section{Examples} \label{sec:examples}
In Table~\ref{ta:1}, we give the values of $B_-(\pi)$ for all permutations of length up to~$4$, and we compare them to the values of $B_+(\pi)$ obtained by~\cite{Elizalde11}. 
Note that more permutations satisfy $B_-(\pi) = 1$ than $B_+(\pi) = 1$; see also the list of open problems at the end of the paper. 
Some other examples are below.

\begin{table}[ht]
\begin{centering}
\begin{tabular}{c|c|c|c}
\hline
$\beta$ & root of & $\pi$ with $B_-(\pi) = \beta$ & $\pi$ with $B_+(\pi) = \beta$ \\
\hline
$1$ & $x - 1$ & $12, 21$ & $12, 21$ \\
& & $123, 132, 213, 231, 321$ & $123, 231, 312$ \\
& & $1324, 1342, 1432, 2134$ & $1234, 2341, 3412, 4123$ \\
& & $2143, 2314, 2431, 3142$ & \\
& & $3214, 3241, 3421, 4213$ & \\
\hline
$1.465$ & $x^3 - x^2 - 1$ & & $1342, 2413, 3124, 4231$ \\
\hline
$1.618$ & $x^2 - x - 1$ & $312$ & $132, 213, 321$ \\
& & $1423, 3412, 4231$ & $1243, 1324, 2431$ \\
& & & $3142, 4312$ \\
\hline
$1.755$ & $x^3 - 2x^2 + x - 1$ & $2341, 2413, 3124, 4123$ & \\
\hline
$1.802$ & $x^3 - 2x^2 - 2x + 1$ & & $4213$ \\
\hline
$1.839$ & $x^3 - x^2 - x - 1$ & $4132$ & $1432, 2143, 3214, 4321$ \\
\hline
$2$ & $x - 2$ & $1234$, $1243$, $4312$ & $2134, 3241$ \\
\hline
$2.247$ & $x^3 - 2x^2 - x + 1$ & $4321$ & $4132$ \\
\hline
$2.414$ & $x^2 - 2x - 1$ & & $2314, 3421$ \\
\hline
$2.618$ & $x^2 - 3x + 1$ & & $1423$ \\
\hline
\end{tabular}
\end{centering}
\caption{$B_-(\pi)$ and $B_+(\pi)$ for all permutations of length up to~4.}
\label{ta:1}
\end{table}

\begin{enumerate}
\item
Let $\pi=3421$. 
Then $n=4$, $\tilde{\pi} = \underline{3}142$, $z_{[1,4)} = 110$, $m=2$, $\pi(n) = 1$, $r = 3$. 
We obtain that $a = \overline{z_{[2,4)}0} = \overline{100} = \overline{\varphi^2(0)}$, thus $B_-(\pi) = b(a) = 1$. 
Indeed, for $w=1\, \overline{100}$, we have $\mathrm{Pat}(w, \Sigma, n) = \pi$ and $\widehat{w} =  a$.

\item 
Let $\pi = 892364157$. 
Then $n=9$, $\tilde{\pi} = 536174\underline{8}92$, $z_{[1,9)} = 33012102$, $m=2$, $\ell=5$, $r=1$, thus $a = z_{[2,9)}\, \overline{z_{[1,9)}} = \overline{30121023}$, and $b(a)$ is the root $x>1$ of 
\[
	x^8-4x^7+x^6-2x^5+3x^4-2x^3+x^2-3x+3=0.
\]
We get $B_-(\pi) \approx 3.831$, and for $w=330121023\,  \overline{301210220}$, we have $\mathrm{Pat}(w, \Sigma, n) = \pi$, $\widehat{w}=\overline{301210220}$ and $b(\widehat{w})=b(a)$.

\item
Let $\pi = 453261$. 
Then $n=6$, $\tilde{\pi} = \underline{4}62531$, $z_{[1,6)} = 11001$, $m=5$, $\pi(n)=1$, $r=4$, thus $a = z_5\, \overline{z_4 z_5} = \overline{10}$, $b(a) = 2$.
For $w=110010\, \overline{2}$, we have $\mathrm{Pat}(w, \Sigma, n) = \pi$, $\widehat{w}=\overline{2}$ and $b(\widehat{w})=b(a)$.

\item
Let $\pi = 7325416$.
Then $n=7$, $\tilde{\pi} = 65214\underline{7}3$, $z_{[1,7)} = 100100$, $m = r = 1$, $\ell = 4$.
Hence $\pi$ is collapsed, and $z^{(0)}_{[1,7)} = 200100$, $z^{(1)}_{[1,7)} = 200210$,  $z^{(2)}_{[1,7)} = 211210$.
Since $n-m$ is even, we have
\[
a = \min_{i\in\{0,1,2\}} z^{(i)}_{[1,7)}\, \overline{z^{(i)}_{[4,7)}} = \min\{200\, \overline{100}, 200\, 
\overline{210}, 211\, \overline{210}\} = 211\, \overline{210}.
\]
Therefore, $B_-(\pi) \approx 2.343$ is the largest positive root of the polynomial $x^6-3x^5+2x^4-x^3-1$.
We have $\mathrm{Pat}(211 (210)^{2k} \overline{2}, \Sigma, n) = \pi$ for all $k\ge 1$ and $\lim_{k\to\infty}b(211 (210)^{2k} \overline{2}) = b(a)$. 

\item 
Let $\tilde{\pi}=4312$. 
Then it can come from one of the following four permutations~$\pi$: $1423$, $3142$, $2314$, $4231$. 
When underlying $\pi(1)$ in $\tilde{\pi}$, one actually differentiates the four possible cases, which we study in detail hereafter.

\begin{itemize}
\item 
Let $\pi=1423$. Hence $\tilde{\pi} = 43\underline{1}2$. 
We find $m=r=2$ and $\ell=3$. 
Then $z_{[1,4)}=000$ and $\pi$ is collapsed.
We get $z^{(0)}_{[1,4)}=010$ and $a = z^{(0)}_{[2,4)}\, \overline{z^{(0)}_{[3,4)}} =1\, \overline{0}$, thus $B_-(\pi) = \frac{1+\sqrt{5}}{2}$. 
We have $\mathrm{Pat}(010^{2k+1}\overline{1}, \Sigma, n) = \pi$ for all $k \ge 0$ and $\lim_{k\to\infty}b(10^{2k+1}\overline{1}) = b(a)$.

\item 
Let $\pi=3142$. Hence $\tilde{\pi} = 4\underline{3}12$. 
We find $m=3$, $r=1$ and $\ell=2$. Then $z_{[1,4)}=001$. 
We get $a = z_{[3,4)}\, \overline{z_{[1,4)}} = \overline{100} = \overline{\varphi^2(0)}$. 
Thus $B_-(\pi)=1$.  
For $w = 00\,\overline{10011}$, we have $\mathrm{Pat}(w, \Sigma, n) =\pi$ and $b(\widehat{w}) = b(\overline{10011}) = 1$. 

\item 
Let $\pi=2314$. Hence $\tilde{\pi}=431\underline{2}$. 
We find $m=4$, $r$ is not defined and $\ell=2$. Then $z_{[1,4)}=000$. 
We get $a = z_{[4,4)}\, \overline{z_{[2,4)}} = \overline{0}$.
Thus $B_-(\pi)=1$.  
We have $\mathrm{Pat}(000\, \overline{1}, \Sigma, n) = \pi$ and $b(\overline{1}) = 1$. 

\item 
Let $\pi=4231$. Hence $\tilde{\pi}=\underline{4}312$. 
We find $m=1$, $r=2$ and $\ell$ is not defined. Then $z_{[1,4)}=100$. 
We get $a = z_{[1,4)}\, \overline{z_{[2,4)}} = 1\, \overline{0}$, $B_-(\pi)=\frac{1+\sqrt{5}}{2}$.
We have $\mathrm{Pat}(10^{2k+1} \overline{1}, \Sigma, n) = \pi$ for $k\ge 1$, $\lim_{k\to\infty}b(10^{2k+1}\overline{1}) = b(a)$. 
\end{itemize}
\end{enumerate}

\section{Permutation patterns of integer sequences}
We first establish a relation between $z_{[1,n)}$ and $w_{[1,n)}$ for sequences $w \in \mathbb{N}^\infty$ satisfying $\mathrm{Pat}(w, \Sigma, n) = \pi$. 
Note $w\in \mathbb{N}^\infty$ realizes the permutation $\pi\in\mathcal{S}_n$, i.e.  
$\mathrm{Pat}(w,\Sigma,n)=\pi$, if and only if 
$w_{[\pi^{-1}(k),\infty)}<w_{[\pi^{-1}(k+1),\infty)}$ for all $1\le k< n$:
\[
	w_{[\pi^{-1}(1),\infty)}<w_{[\pi^{-1}(2),\infty)}<\ldots<w_{[\pi^{-1}(n),\infty)}.
\]

\begin{proposition} \label{p:pat}
Let $w = w_1w_2 \cdots \in \mathbb{N}^\infty$ and $\pi\in\mathcal{S}_n$.
Then $\mathrm{Pat}(w, \Sigma, n) = \pi$ if and only if the following conditions (\ref{i:1})--(\ref{i:3}) hold. 
\renewcommand{\theenumi}{\roman{enumi}}
\begin{enumerate}
\item \label{i:1}
$w_j - w_i \ge z_j - z_i$ for all $1 \le i, j < n$ with $\pi(j) > \pi(i)$, \\ in particular $w_j \ge z_j$ for all $1 \le j < n$,
\item \label{i:2}
$w_{[n,\infty)} > w_{[\ell,\infty)}$ if $\pi(n) \ne 1$,
\item \label{i:3}
$w_{[n,\infty)} < w_{[r,\infty)}$ if $\pi(n) \ne n$.
\end{enumerate}
\end{proposition}

In the proof of Proposition~\ref{p:pat}, we use the following two lemmas. 

\begin{lemma}\label{running-argument-0}
Let $1\le i,j <n$. If $\pi(i)<\pi(j)$ then $z_i\le z_j$, and $z_i = z_j$ implies that $\pi(i+1)>\pi(j+1)$.
\end{lemma}

\begin{proof}
Let $1\le i,j <n$ with $\pi(i)<\pi(j)$.
Then $z_i\le z_j$ is clear from the definition of the digits~$z_j$.
Moreover, $z_i = z_j$ implies that $\pi(\pi^{-1}(k)+1) > \pi(\pi^{-1}(k+1)+1)$ for all $k \notin \{\pi(n)-1,\pi(n)\}$ with $\pi(i) \le k < \pi(j)$, as well as $\pi(\pi^{-1}(\pi(n)-1)+1) > \pi(\pi^{-1}(\pi(n)+1)+1)$ if $\pi(i) < \pi(n) < \pi(j)$. 
This gives that $\pi(i+1) > \pi(j+1)$. 
\end{proof}

\begin{lemma} \label{l:running-argument}
Assume that $w_{[1,n)}$ satisfies point~(\ref{i:1}) of Proposition~\ref{p:pat}, let $1 \le i < j< n$.

If $\pi(i) < \pi(j)$, then $w_{[i,i+n-j)} \le w_{[j,n)}$, with $(-1)^{n-j} \pi(i+n-j) < (-1)^{n-j} \pi(n)$ in case of equality.

If $\pi(i) > \pi(j)$, then $w_{[i,i+n-j)} \ge w_{[j,n)}$, with $(-1)^{n-j} \pi(i+n-j) > (-1)^{n-j} \pi(n)$ in case of equality.
\end{lemma}

\begin{proof}
If $\pi(i) < \pi(j)$, then we have $z_i \le z_j$, thus $w_i \le w_j$, and equality implies that $\pi(i+1) > \pi(j+1)$ by Lemma~\ref{running-argument-0}.
Similarly, $\pi(i) > \pi(j)$ implies that $w_i > w_j$ or $w_i = w_j$, $\pi(i+1) < \pi(j+1)$.
Recursively, we get the statement of the lemma.
\end{proof}

\begin{proof}[Proof of Proposition~\ref{p:pat}]
Assume first that $\mathrm{Pat}(w, \Sigma, n) = \pi$. 
Then (\ref{i:2}) and (\ref{i:3}) hold immediately.
For $1 \le i, j < n$ with $\pi(j) > \pi(i)$, we use the telescoping sum
\[
z_j - z_i = \sum_{\substack{\pi(i) \le k < \pi(j), \\ k \notin \{\pi(n)-1,\pi(n)\}}} 	
                \big(z_{\pi^{-1}(k+1)} - z_{\pi^{-1}(k)}\big) 
		+	\begin{cases}
		 	z_r - z_\ell	& \mbox{if } \pi(i) < \pi(n) < \pi(j), \\ 
		 	0		 	& \mbox{otherwise}.
			\end{cases}
\]
This relation is also valid with $z$ replaced by~$w$.
To prove that $w_j-w_i \ge z_j-z_i$, it is therefore sufficient to show that $w_r - w_\ell \ge z_r - z_\ell$ if $\pi(n) \notin \{1,n\}$, and
\begin{equation} \label{e:wz}
w_{\pi^{-1}(k+1)} - w_{\pi^{-1}(k)} \ge z_{\pi^{-1}(k+1)} - z_{\pi^{-1}(k)}
\end{equation}
for all $1 \le k < n$,  with $k \notin \{\pi(n)-1,\pi(n)\}$.
Since $w_{[\pi^{-1}(k+1),\infty)} > w_{[\pi^{-1}(k),\infty)}$, we have $w_{\pi^{-1}(k+1)} > w_{\pi^{-1}(k)}$ or $w_{\pi^{-1}(k+1)} = w_{\pi^{-1}(k)}$, $w_{[\pi^{-1}(k+1)+1,\infty)} < w_{[\pi^{-1}(k)+1,\infty)}$.
The latter inequality means that $\pi(\pi^{-1}(k+1)+1) < \pi(\pi^{-1}(k)+1)$, i.e., $z_{\pi^{-1}(k+1)} = z_{\pi^{-1}(k)}$.
As $z_{\pi^{-1}(k+1)} - z_{\pi^{-1}(k)} \in \{0,1\}$, this proves~\eqref{e:wz}.
Similarly, if $\pi(n) \notin \{1,n\}$, then $w_{[\pi^{-1}(\pi(n)+1),\infty)} > w_{[\pi^{-1}(\pi(n)-1),\infty)}$ implies that $w_r - w_\ell \ge 1 \ge z_r - z_\ell$ 
or $w_r = w_\ell$, $w_{[r+1,\infty)} < w_{[\ell+1,\infty)}$; in the latter case, we have $\pi(r+1) < \pi(\ell+1)$ and thus $z_r = z_\ell$. 
The inequality $w_j \ge z_j$ follows by setting $i = \pi^{-1}(1)$, as $z_i = 0$ and thus $w_j \ge z_j + w_i \ge z_j$ in this case.
This proves (\ref{i:1}). 

\smallskip
Conversely, assume in the following that conditions (\ref{i:1})--(\ref{i:3}) hold. 
We claim that 
\begin{equation} \label{e:win}
\begin{cases}
w_{[i,\infty)} < w_{[n,\infty)} & \mbox{for all $i$ with $\pi(i) < \pi(n)$}, \\
w_{[i,\infty)} > w_{[n,\infty)} & \mbox{for all $i$ with $\pi(i) > \pi(n)$}.
\end{cases}
\end{equation}

Suppose that the claim is false. 
If $w_{[i,\infty)} \ge w_{[n,\infty)}$ for some~$i$ with $\pi(i) < \pi(n)$, then we have $\pi(n) \ne 1$, $w_{[i,\infty)} > w_{[\ell,\infty)}$ by~(\ref{i:2}), and $\pi(i) < \pi(\ell) = \pi(n)-1$.
For~$i < \ell$, Lemma~\ref{l:running-argument} gives that $w_{[i,i+n-\ell)} \le w_{[\ell,n)} \le w_{[i,i+n-\ell)}$, thus $w_{[i,i+n-\ell)} = w_{[\ell,n)}$ and
\[
\begin{cases}
\pi(i+n-\ell) < \pi(n),\ w_{[i+n-\ell,\infty)} > w_{[n,\infty)}
			& \mbox{if $n-\ell$ is even}, \\ 
\pi(i+n-\ell) > \pi(n),\ w_{[i+n-\ell,\infty)} < w_{[n,\infty)} 		
			& \mbox{if $n-\ell$ is odd}.
\end{cases}
\]
Moreover, if $w_{[i,i+k)} > w_{[n,n+k)} (\ge w_{[\ell,\ell+k)})$, then $k>n-\ell$, $w_{[i+n-\ell,i+k)} > w_{[n,\ell+k)}$ if $n-\ell$ is even, $w_{[i+n-\ell,i+k)} < w_{[n,\ell+k)}$  if $n-\ell$ is odd. 
For $i > \ell$, we obtain in the same way that $w_{[i,n)} = w_{[\ell,\ell+n-i)}$ and
\[
\begin{cases}
\pi(n) < \pi(\ell+n-i),\ w_{[n,\infty)} > w_{[\ell+n-i,\infty)}	
			& \mbox{if $n-i$ is even}, \\ 
\pi(n) > \pi(\ell+n-i),\ w_{[n,\infty)} < w_{[\ell+n-i,\infty)}	
			& \mbox{if $n-i$ is odd}.
\end{cases}
\]
Here, $w_{[i,i+k)} > w_{[n,n+k)} (\ge w_{[\ell,\ell+k)})$ implies that $k>n-i$, $w_{[n,i+k)} > w_{[\ell+n-i,\ell+k)}$ if $n-i$ is even,  $w_{[n,i+k)} < w_{[\ell+n-i,\ell+k)}$  if $n-i$ is odd. 
If $w_{[i,\infty)} \le w_{[n,\infty)}$ for some~$i$ with $\pi(i) > \pi(n)$, then we get the opposite inequalities, with $\ell$ replaced by~$r$.
In particular, we have some~$i$ such that 
\begin{equation} \label{e:win2}
w_{[i,\infty)} > w_{[n,\infty)},\ \pi(i) < \pi(n), \quad \mbox{or} \quad w_{[i,\infty)} < w_{[n,\infty)},\ \pi(i) > \pi(n).
\end{equation}
Let $k \ge 1$ be minimal such that there is some $i$ with
\begin{equation} \label{e:wink}
w_{[i,i+k)} > w_{[n,n+k)},\ \pi(i) < \pi(n), \quad \mbox{or} \quad w_{[i,i+k)}  < w_{[n,n+k)},\ \pi(i) > \pi(n).
\end{equation}
Then the above arguments give 
\begin{equation} \label{e:wink2}
w_{[j,j+h)} > w_{[n,n+h)},\ \pi(j) < \pi(n), \quad \mbox{or} \quad w_{[j,j+h)} < w_{[n,n+h)},\ \pi(j) > \pi(n),
\end{equation}
for some $j \in \{n-|i-\ell|, n-|i-r|\}$, $h \in \{k-n+\ell, k-n+i, k-n+r\}$, contradicting the minimality of~$k$. 
Hence \eqref{e:win} holds. 

For $1 \le i < j \le n$ with $\pi(i) < \pi(j)$, we obtain that
\[
w_{[i,\infty)} = w_{[i,i+n-j)} w_{[i+n-j,\infty)} < w_{[j,n)} w_{[n,\infty)} = w_{[j,\infty)},
\]
as $w_{[i,i+n-j)} < w_{[j,n)}$ or $w_{[i,i+n-j)} = w_{[j,n)}$, $(-1)^{n-j} \pi(i+n-j) < (-1)^{n-j} \pi(n)$, by Lemma~\ref{l:running-argument}, and the latter inequality implies $w_{[i+n-j,\infty)}<w_{[n,\infty)}$ if $n-j$ is even, $w_{[i+n-j,\infty)} > w_{[n,\infty)}$ if $n-j$ is odd, by~\eqref{e:win}.
Similarly, $w_{[i,\infty)} > w_{[j,\infty)}$ holds for $1 \le i < j \le n$ with $\pi(i) > \pi(j)$. 
Hence $\mathrm{Pat}(w, \Sigma, n) = \pi$.
\end{proof}

\begin{remark} \label{r:pat}
If $\pi(n) \ne 1$, then $w_{[n,\infty)} > w_{[\ell,\infty)}$ is equivalent to $w_{[n,\infty)} > \overline{w_{[\ell,n)}}$.
Indeed, suppose that $w_{[\ell,\infty)} <w_{[n,\infty)} \le \overline{w_{[\ell,n)}}$ or  $\overline{w_{[\ell,n)}} < w_{[n,\infty)} \le w_{[\ell,\infty)}$.
Then $w_{[\ell,n)} = w_{[n,2n-\ell)} = w_{[2n-\ell,3n-2\ell)} = \cdots$, hence $w_{[\ell,\infty)} = \overline{w_{[\ell,n)}}$, a contradiction.
Similarly, $w_{[n,\infty)} < w_{[r,\infty)}$ is equivalent to $w_{[n,\infty)} < \overline{w_{[r,n)}}$ if $\pi(n) \ne n$.
Hence we can replace $w_{[\ell,\infty)}$ by $\overline{w_{[\ell,n)}}$ in Proposition~\ref{p:pat}~(\ref{i:2}), $w_{[r,\infty)}$ by $\overline{w_{[r,n)}}$ in Proposition~\ref{p:pat}~(\ref{i:3}).
\end{remark}
 
If $\pi$ is collapsed, then we have to increase some digits of $z_{[1,n)}$ to obtain a sequence $w \in \mathbb{N}^\infty$ with $\mathrm{Pat}(w, \Sigma, n) = \pi$.

\begin{lemma} \label{l:zi}
Let $\pi \in \mathcal{S}_n$ be collapsed, $w = w_1w_2 \cdots \in \mathbb{N}^\infty$ such that $\mathrm{Pat}(w, \Sigma, n)\, {=}\, \pi$. 
Then $w_m \ge z_m + 1$, with equality if and only if $w_{[1,n)} = z^{(i)}_{[1,n)}$ for some $0 \le i < |r-\ell|$.
\end{lemma}

\begin{proof}
By Proposition~\ref{p:pat} (\ref{i:2})--(\ref{i:3}) and Remark~\ref{r:pat}, we have $\overline{w_{[\ell,n)}} < \overline{w_{[r,n)}}$, hence the collapsedness implies that $w_{[\ell,n)} \ne z_{[\ell,n)}$ or $w_{[r,n)} \ne z_{[r,n)}$.
By Proposition~\ref{p:pat}~(\ref{i:1}), we have thus $w_i > z_i$ for some $1 \le i < n$, and $w_m \ge z_m + w_i - z_i > z_m$, as $m = \pi^{-1}(n)<n$ (because $\pi$ is collapsed). 
If $w_{[1,n)} = z^{(i)}_{[1,n)}$ for some $0 \le i < |r-\ell|$, then we have $w_m = z^{(i)}_m = z_m + 1$.

Assume now that $w_m = z_m + 1$.
From Proposition~\ref{p:pat}~(\ref{i:1}), we get $w_j \in \{z_j,z_j+1\}$ for all $1 \le j < n$.
We have $w_{[\ell,\ell+|r-\ell|)} < w_{[r,r+|r-\ell|)}$, as $w_{[\ell,\ell+|r-\ell|)} = w_{[r,r+|r-\ell|)}$ would imply that $\overline{w_{[\ell,n)}} = \overline{w_{[r,n)}}$.
Since $z_{[\ell,\ell+|r-\ell|)} = z_{[r,r+|r-\ell|)}$, we obtain that $w_{r+i} - w_{\ell+i} =(-1)^i$ for some $0 \le i < |r-\ell|$, thus $w_{\ell+i} = z_{\ell+i}^{(i)}$ and $w_{r+i} = z_{r+i}^{(i)}$.
By Proposition~\ref{p:pat} and its proof, $w_m = z_m+1$ implies that exactly one of the differences $(w_r - w_\ell) - (z_r - z_\ell)$ and $(w_{\pi^{-1}(k+1)} - w_{\pi^{-1}(k)}) - (z_{\pi^{-1}(k+1)} - z_{\pi^{-1}(k)})$, $1 \le k < n$, $k \notin \{\pi(n)-1,\pi(n)\}$, equals~$1$ and all others are~$0$. 
If $i = 0$, then we obtain that  $(w_r - w_\ell) - (z_r - z_\ell) = 1$.
This implies $w_j=z_j=z_j^{(0)}$ for all $1\le j<n$ with $j\ne r$, thus $w_{[1,n)} = z^{(0)}_{[1,n)}$. 
Assume in the following that $1 \le i < |r-\ell|$.
Suppose that $\pi(\ell+i)$ and $\pi(r+i)$ are not consecutive integers, i.e., $\pi(j)$ is between $\pi(\ell+i)$ and $\pi(r+i)$ for some~$j$.
Then Lemma~\ref{l:running-argument} gives that $\pi(j-i+|r-\ell|)$ is between $\pi(\ell+|r-\ell|)$ and $\pi(r+|r-\ell|)$, contradicting that the pair $(\ell+|r-\ell|,r+|r-\ell|)$ is either $(r,n)$ or $(n,\ell)$.
Therefore, we have $\{\pi(r+i),\pi(\ell+i)\} = \{k,k+1\}$ for some $1 \le k < n$, $k \notin \{\pi(n)-1,\pi(n)\}$.
Then $(w_{\pi^{-1}(k+1)} - w_{\pi^{-1}(k)}) - (z_{\pi^{-1}(k+1)} - z_{\pi^{-1}(k)}) = 1$ for this~$k$, $w_{\pi^{-1}(k+1)} - w_{\pi^{-1}(k)} = z_{\pi^{-1}(k+1)} - z_{\pi^{-1}(k)}$ for all other $k \notin \{\pi(n)-1,\pi(n)\}$, and $w_r - w_\ell = z_r - z_\ell$. 
This implies that $w_{[1,n)} = z^{(i)}_{[1,n)}$. 
\end{proof}

Let us illustrate the previous proof by again considering the (collapsed) permutation $\pi= 7325416$.
Recall that $z_{[1,7)} = 100100$, $m = r = 1$, $\ell = 4$.
Choose $w \in \mathbb{N}^\infty$ with $\mathrm{Pat}(w, \Sigma, 7) = \pi$.
In order to satisfy $w_m=z_m+1=2$, we get from Proposition~\ref{p:pat}~(\ref{i:1})
that the prefix $w_{[1,7)}$ must be one of the following six sequences: $200100,\,  200200,\,  200210,\, 210210,\, 211210,\, 211211$. 
But from Proposition~\ref{p:pat} (\ref{i:2})--(\ref{i:3}) and Remark~\ref{r:pat}, the prefixes $200200,\, 210210$ and $211211$ are not possible, only $z^{(0)}_{[1,7)} = 200100$,  $z^{(1)}_{[1,7)} = 200210$ and  $z^{(2)}_{[1,7)} = 211210$ are possible.

The following lemma shows that $b(a)$ is well defined. 

\begin{lemma} \label{l:a}
We have $\widehat{a} = a$.
If $\pi(n) = 1$, then $\overline{0 z_{[m,n)}} \le \overline{z_{[r,n)}}$. 
If $\pi(n) \notin \{1,n\}$, then $\overline{z_{[\ell,n)}} \le \overline{z_{[r,n)}}$. 
\end{lemma}

\begin{proof}
To prove that $\widehat{a} = a$, we show that $\widehat{w} = w_{[m,\infty)}$ for all sequences~$w$ satisfying $w_{[1,n)} = z_{[1,n)}$ or $w_{[1,n)} = z^{(i)}_{[1,n)}$ for some $0 \le i < |r-\ell|$, if $\pi(n) \notin \{1,n\}$, $w_{[n,\infty)} = \overline{w_{[\ell,n)}}$, if $\pi(n) \ne 1$, or $w_{[n,\infty)} = \overline{w_{[r,n)}}$, if $\pi(n) \ne n$, or $w_{[n,\infty)} = \overline{0w_{[m,n)}}$, if $\pi(n) = 1$. 
(This means that $w_{[n,\infty)} \in \{w_{[\ell,\infty)}, w_{[r,\infty)}, 0w_{[m,\infty)}\}$.)

We first claim that $w_{[i,\infty)} \le w_{[n,\infty)}$ for all~$i$ with $\pi(i) < \pi(n)$, $w_{[i,\infty)} \ge w_{[n,\infty)}$ for all~$i$ with $\pi(i) > \pi(n)$.
The proof is similar to that of~\eqref{e:win}.
Note that condition~(\ref{i:1}) of Propositon~\ref{p:pat} holds.
Suppose that the claim is false, i.e., \eqref{e:win2} holds for some~$i$. 
Let $k \ge 1$ be minimal such that \eqref{e:wink} holds for some~$i$. 
If $w_{[n,\infty)} = w_{[\ell,\infty)}$, then we have $w_{[i,i+k)} > w_{[\ell,\ell+k)}$, $\pi(i) < \pi(\ell)$, or $w_{[i,i+k)} < w_{[\ell,\ell+k)}$, $\pi(i) > \pi(\ell)$, thus Lemma~\ref{l:running-argument} gives that \eqref{e:wink2} holds for $j = n-|i-\ell|$, $h  = k-n+\ell$ or $h = k-n+i$, contradicting the minimality of~$k$.
For $w_{[n,\infty)} = w_{[r,\infty)}$, the same arguments apply, with $\ell$ replaced by~$r$. 
If $w_{[n,\infty)} = 0 w_{[m,\infty)}$ and $\pi(n) = 1$, then we have $w_{[i,i+k)} < 0 w_{[m,m+k-1)}$, thus $w_{[i+1,i+k)} > w_{[m,m+k-1)}$, with $\pi(i+1) < n = \pi(m)$. 
Now, \eqref{e:wink2} holds for $j = n-|i+1-m|$, $h  = k-n+m$ or $h = k-n+i$, contradicting again the minimality of~$k$.
This proves the claim. 

Similarly to the last paragraph of the proof of Propositon~\ref{p:pat}, we obtain for $1 \le i,j < n$ that $w_{[i,\infty)} \le w_{[j,\infty)}$ if $\pi(i) < \pi(j)$, thus $\widehat{w} = \max_{1\le i < n} w_{[i,\infty)} = w_{[m,\infty)}$. 
Note that if $m=n$, then $\pi(n)=n$ and $w_{[m,\infty)}=w_{[\ell,\infty)}$.
This implies that $\widehat{a} = a$. 
 
For $\pi(n) \notin \{1,n\}$, we have seen above that $w_{[n,\infty)} \le w_{[r,\infty)}$ for $w_{[1,n)} = z_{[1,n)}$ and $w_{[n,\infty)} = \overline{z_{[\ell,n)}}$, thus $w_{[n,\infty)} \le \overline{w_{[r,n)}}$ by Remark~\ref{r:pat}, i.e., $\overline{z_{[\ell,n)}} \le \overline{z_{[r,n)}}$.
In the same way, taking $w_{[n,\infty)} = \overline{0 z_{[m,n)}}$ gives that $\overline{0 z_{[m,n)}} \le \overline{z_{[r,n)}}$ for $\pi(n) = 1$.
\end{proof}

The next lemma justifies the definition of collapsedness. 
Here, a~finite word~$v$ is \emph{primitive} if it is not the power of another word, i.e., if $v = s^k$ implies that $s = v$, $k = 1$. 
We say that $v$ is \emph{almost primitive} if $v = s^k$ implies that $k = 1$, or $k = 2$ and $s$ has odd length. 
The length of a finite word~$v$ is denoted by~$|v|$. 

\begin{lemma} \label{l:period}
Assume that $w_{[1,n)}$ satisfies point~(\ref{i:1}) of Propositon~\ref{p:pat}.
If $\pi(n) = 1$ and $n-m$ is even, then $w_{[m,n)} 0$ is primitive.
If $\pi(n) \ne 1$, then $w_{[\ell,n)}$ is almost primitive. 
If $\pi(n) \ne n$, then $w_{[r,n)}$ is almost primitive. 
In particular, for $\pi(n) \notin \{1,n\}$, we have $\overline{z_{[\ell,n)}} = \overline{z_{[r,n)}}$ if and only if $\pi$ is collapsed.
\end{lemma}

\begin{proof}
Let first $n-m$ be even, and suppose that $w_{[m,n)} 0 = s^k$ for some word~$s$ and some $k \ge 2$.
Then $|s|$ is odd, we have $\pi(m) = n > \pi(m+|s|)$ and $w_{[m,n-|s|)} = w_{[m+|s|,n)}$, thus $\pi(n-|s|) < \pi(n)$ by Lemma~\ref{l:running-argument} (as $n-m-|s|$ is odd).
If $\pi(n) = 1$, then this is impossible, hence $w_{[m,n)}0$ is primitive.
 
Let now $\pi(n) \ne 1$, and let $p \ge 1$ be minimal such that $p$ divides $n-\ell$ and $w_{[\ell,n)} = (w_{[\ell,\ell+p)})^{(n-\ell)/p}$. 
By Lemma~\ref{l:running-argument}, we have, for $1 \le i < j \le (n-\ell)/p$,
\[
\mathrm{sgn}\big(\pi(\ell+jp) - \pi(\ell+ip)\big) = (-1)^{ip} \mathrm{sgn}\big(\pi(\ell+jp-ip) - \pi(\ell)\big).
\]
We distinguish the following cases:
\begin{itemize}
\item
If $p$ is even and $\pi(\ell) < \pi(\ell+p)$, then we get that $\pi(\ell) < \pi(\ell+p) < \pi(\ell+2p) < \cdots < \pi(n)$. 
Since $\pi(\ell)=\pi(n)-1$, we get $n = \ell+p$. 
\item
If $p$ is even and $\pi(\ell) > \pi(\ell+p)$, then we have $\pi(\ell) > \pi(\ell+p) > \pi(\ell+2p) > \cdots > \pi(n)$, which is impossible.
\item
If $p$ is odd and $\pi(\ell) < \pi(\ell+2p)$ (if $n-\ell \ge 2p$), then we obtain that $\pi(\ell) < \pi(\ell+2p) < \pi(\ell+4p) < \cdots < \pi(\ell + \lfloor \frac{n-\ell}{2p}\rfloor 2p)$.
Therefore, $n = \ell+2p$ or $(n-\ell)/p$ is odd. 
If $(n-\ell)/p$ is odd, then we get that $\pi(\ell+p) > \pi(\ell+3p) > \cdots > \pi(n)$, thus $\pi(\ell+p) > \pi(\ell)$.
This implies that $\pi(n) > \pi(n-p)$, and we know from above that $\pi(n-p) \ge \pi(\ell)$, hence $n = \ell+p$. 
\item
If $p$ is odd and $\pi(\ell) > \pi(\ell+2p)$ (if $n-\ell \ge 2p$), then $\pi(\ell) > \pi(\ell+2p) > \pi(\ell+4p) > \cdots > \pi(\ell + \lfloor \frac{n-\ell}{2p}\rfloor 2p)$, thus $(n-\ell)/p$ is odd. 
Now, $\pi(\ell+p) < \pi(\ell)$ is impossible since this would imply that $\pi(n) < \pi(n-p) \le \pi(\ell)$.
Therefore, we have $\pi(\ell) < \pi(\ell+p) < \pi(\ell+3p) < \cdots < \pi(n)$, thus $n = \ell + p$.
\end{itemize}
The proof for $\overline{w_{[r,n)}}$ is symmetric. 

If $\ell < r$ and $\overline{z_{[\ell,n)}} = \overline{z_{[r,n)}}$, then the almost primitivity of $z_{[\ell,n)}$ gives $z_{[\ell,n)} = z_{[r,n)} z_{[r,n)}$, with $|n-r|$ odd. 
Similarly, $\ell > r$ and $\overline{z_{[\ell,n)}} = \overline{z_{[r,n)}}$ imply that $z_{[r,n)} = z_{[\ell,n)} z_{[\ell,n)}$, with $|n-\ell|$ odd.  
Thus $\pi$ is collapsed if and only if $\overline{z_{[\ell,n)}} = \overline{z_{[r,n)}}$. 
\end{proof}

\section{Characterization of $(-\beta)$-shifts}
We determine for a given sequence to which $(-\beta)$-shifts it belongs. 
In the following proposition, which is proved at the end of the section, we use the notation
\[
v' = \begin{cases}v_1 v_2 \cdots v_{j-1} (v_j{-}1) 0 & \mbox{if}\ v_j \ne 0, \\ v_1 v_2 \cdots v_{j-2} (v_{j-1}{+}1) & \mbox{if}\ v_j = 0,\end{cases}
\]
for $v = v_1 v_2 \cdots v_j \in \mathbb{N}^+ \setminus \{0\}$, where $\mathbb{N}^+$ denotes the set of non-empty finite words of non-negative integers. 
Then we have $\overline{v} < \overline{v'}$ if $|v|$ is even, $\overline{v} > \overline{v'}$ if $|v|$ is odd.

\begin{proposition} \label{p:Omega}
Let $w \in \mathbb{N}^\infty$ be a bounded sequence. 
Then we have $w \in \Omega_{-\beta}$ for all $\beta > b(\widehat{w})$ and $w \notin \Omega_{-\beta}$ for all $1 < \beta < b(\widehat{w})$. 

If $b(\widehat{w}) > 1$, then we have $w \in \Omega_{-b(\widehat{w})}$ if and only if $w$ does not end with~$0 \widehat{w}$, and $\widehat{w} = d_{-b(\widehat{w})}(1)$ or $\widehat{w} = \overline{v'}$, $d_{-b(\widehat{w})}(1) = \overline{v}$ with $|v|$ odd, $v$ primitive.
\end{proposition}

For $\beta > 1$, let $W_{-\beta}$ be the set of sequences $w \in \mathbb{N}^\infty$ such that $\widehat{w} = w$,
\begin{equation} \label{e:bw}
- \sum_{j=1}^\infty \frac{w_j+1}{(-\beta)^j} = 1 \quad \mbox{and} \quad - \sum_{j=1}^\infty \frac{w_{k+j}+1}{(-\beta)^j} \in [0,1] \quad \mbox{for all}\ k \ge 1.
\end{equation}
By Corollary~1 of~\cite{Steiner13}, for each $w \in \mathbb{N}^\infty$ with $\widehat{w} = w > u$, there is a unique $\beta > 1$ such that $w \in W_{-\beta}$. 
Let $W_{-1}$ be the set of sequences $w \in \mathbb{N}^\infty$ such that $\widehat{w} = w \le u$.

\begin{lemma} \label{l:Wbeta}
For $1 \le \alpha < \beta$, we have $W_{-\alpha} \cap W_{-\beta} = \emptyset$.
If $w \in W_{-1}$, then $w = \overline{\varphi^k(0)}$ for some $k \ge 0$ or $w = u$. 
If $w \in W_{-\beta}$, $\beta > 1$, then $w > u$. 
\end{lemma}

\begin{proof}
By Theorem~1 of~\cite{NguemaNdong16}, $w \in W_{-1}$ implies that $w = \overline{\varphi^k(0)}$ for some $k \ge 0$ or $w = u$. 
(Note that $\phi(1)$ should be $\phi^\infty(1)$ in Nguema Ndong's theorem.) 
For $w = \overline{\varphi^k(0)}$, $k \ge 1$, we have ${-} \sum_{j=1}^\infty \frac{w_j+1}{(-\beta)^j} \ne 1$ for all $\beta > 1$ by Lemma~3.4 of~\cite{Liao-Steiner12}.
By Proposition~3.5 of~\cite{Liao-Steiner12}, we have $u \notin W_{-\beta}$ for all $\beta > 1$.
This implies that $W_{-1} \cap W_{-\beta} = \emptyset$ for all $\beta > 1$, in particular $w > u$ for all $w \in W_{-\beta}$, $\beta > 1$. 
Hence, we have $W_{-\alpha} \cap W_{-\beta} = \emptyset$ for distinct $\alpha, \beta > 1$ by Corollary~1 of~\cite{Steiner13}. 
\end{proof}

The set $W_{-\beta}$ is related to $d_{-\beta}(1)$ in the following way. 
Here, $\{v,v'\}^\infty$ dentoes the set of all infinite concatenations of copies of $v$ and~$v'$. 
We use the polynomials
\[
P_{v_1v_2\cdots v_j}(x) = (-x)^j + \sum_{k=1}^j (v_k+1) (-x)^{j-k}.
\]

\begin{lemma} \label{l:ibeta}
We have $d_{-\beta}(1) \in W_{-\beta}$ for all $\beta > 1$. 
If $d_{-\beta}(1)$ is not purely periodic, then $W_{-\beta} = \{d_{-\beta}(1)\}$.
If $d_{-\beta}(1) = \overline{v}$, $v$~primitive, then $w \in W_{-\beta}$ is equivalent to $w \in \{v,v'\}^\infty$ and $\widehat{w} = w$.
Moreover, $v$~does not end with~$0$, $v'$~is primitive, and $\overline{v'} \in W_{-\beta}$ if $|v|$ is odd. 
\end{lemma}

\begin{proof}
Let $\beta>1$. 
From $d_{-\beta}(1)\in\Omega_{-\beta}$, we obtain that $d_{-\beta}(1) \in W_{-\beta}$.
For $w \in \mathbb{N}^\infty$, note that ${-}\sum_{j=1}^\infty \frac{w_j+1}{(-\beta)^j} = 1$ implies that ${-} \sum_{j=1}^\infty \frac{w_{k+j}+1}{(-\beta)^j} = P_{w_{[1,k]}}(\beta)$ for all $k \ge 1$.
We also have $P_{d_{-\beta,1}(1)\cdots d_{-\beta,k}(1)}(\beta) = T_{-\beta}^k(1) \in (0,1]$ for all $k \ge 1$. 
If $d_{-\beta}(1)$ is not purely periodic, then $T_{-\beta}^k(1) \ne 1$ for all $k \ge 1$, thus \eqref{e:bw} holds if and only if $w = d_{-\beta}(1)$. 
If $d_{-\beta}(1) = \overline{v}$, then $P_v(\beta) = 1$, thus $v$ does not end with~$0$, we have $P_{v_{[1,|v|-1]}(v_{|v|}{-}1)}(\beta) = 0$ and $P_{v'}(\beta) = 1$. 
For $w \in \{v,v'\}^\infty$, we obtain that $P_{w_{[1,k]}}(\beta) \in [0,1]$ for all $k \ge 1$, thus 
\[
1 + \sum_{j=1}^\infty \frac{w_j+1}{(-\beta)^j} = \lim_{k\to\infty} 1 + \sum_{j=1}^k \frac{w_j+1}{(-\beta)^j} = \lim_{k\to\infty} \frac{P_{w_{[1,k]}}}{(-\beta)^k} = 0,
\]
and \eqref{e:bw} holds by the first paragraph of the proof. 
If $v$ is primitive, then we have $T_{-\beta}^k(1) \ne 1$ for all $1 \le k < |v|$, hence $w \in W_{-\beta}$ implies that $w \in \{v,v'\}^\infty$.
Suppose that $v' = s^k$ for some $k \ge 2$. 
As $\overline{v} > u$, we have $s \ne 0$ and thus $v = s^{k-1} s'$, contradicting Theorem~2 of~\cite{Steiner13}.
If $|v|$ is odd, then we have $\overline{v'} = \lim_{x\to 1} d_{-\beta}(x)$ by Lemma~6 of~\cite{Ito-Sadahiro09}, thus $\widehat{\overline{v'}} = v'$ and $\overline{v'} \in W_{-\beta}$.
\end{proof}

We also have $\overline{v'} \in W_{-\beta}$ if $|v|$ is even in Lemma~\ref{l:ibeta}. 
Indeed, it can be shown, for any almost primitive word $v \in \mathbb{N}^+ \setminus \{0\}$ with $\widehat{\overline{v}} = \overline{v}$, that $\widehat{\overline{v'}} = \overline{v'}$ and $v'$ is almost primitive. 
The condition $w \in \{v,v'\}^\infty$ in Lemma~\ref{l:ibeta} can be replaced by inequalities. 

\begin{lemma} \label{l:concatenation}
Let $v \in \mathbb{N}^+ \setminus \{0\}$, $w \in \mathbb{N}^\infty$ with $\widehat{w} = w$.
Then $w \in \{v,v'\}^\infty$ if and only if $\overline{v} \le w \le v' \overline{v}$ when $|v|$ is even, $\overline{v'} \le w\le v\, \overline{v'}$ when $|v|$ is odd.
\end{lemma}

\begin{proof}
If $|v|$ is even (resp.\ odd), then $v$ has a prefix that is smaller (resp.\ larger) than a prefix of~$v'$ of same length. 
Therefore, $w \in \{v,v'\}^\infty$ implies that $\overline{v} \le w \le v'\,\overline{v}$ when $|v|$ is even, i.e., $|v'|$ is odd, $\overline{v'} \le w \le v\, \overline{v'}$ when $|v|$ is odd, i.e., $|v'|$ is even. 

Assume now that $\overline{v} \le w \le v' \overline{v}$, $|v|$~even, or $\overline{v'} \le w\le v\, \overline{v'}$, $|v|$~odd.
Then $w$ starts with $v$ or~$v'$.
If $|v|$ is even, then $\widehat{w} = w$ implies that $w_{[i,\infty)} \le w \le v' \overline{v}$ for all $i \ge 1$. 
If $w_{[1,i)} = v^k$ for some $k \ge 0$, then we also have $w_{[i,\infty)} \ge \overline{v}$, thus $w_{[i,\infty)}$ starts with $v$ or~$v'$.
If $w_{[1,i)} = w_{[1,j)} v' v^k$ for some $k \ge 0$, $j \ge 1$, then $w_{[j,\infty)} \le v'\, \overline{v}$ implies that $w_{[i,\infty)} \ge \overline{v}$, and $w_{[i,\infty)}$ starts again with $v$ or~$v'$.
Hence, we obtain that $w \in \{v,v'\}^\infty$.
For odd $|v|$, it suffices to exchange $v$ and~$v'$ in these arguments. 
\end{proof}

\begin{lemma} \label{l:b-increasing}
Let $1 \le \alpha < \beta$, $s \in W_{-\alpha}$ and $w \in W_{-\beta}$.
Then we have $s<w$.
In particular, we have $s < d_{-\beta}(1)$, and $s < \overline{v'}$ if $d_{-\beta}(1) = \overline{v}$, $v$~primitive, thus $s \in \Omega_{-\beta}$. 
\end{lemma}

\begin{proof}
We have already seen in Lemma~\ref{l:Wbeta} that $W_{-\alpha} \cap W_{-\beta} = \emptyset$.
If $\alpha = 1$, then we have $s \le u < w$. 
If $\alpha > 1$, then we have $d_{-\alpha}(1) < d_{-\beta}(1)$ by Theorem~3 in~\cite{Steiner13}.
As $d_{-\alpha}(1) \in W_{-\alpha}$, $d_{-\beta}(1) \in W_{-\beta}$, and the elements of $W_{-\alpha}$ and $W_{-\beta}$ respectively are contiguous by Lemmas~\ref{l:ibeta} and~\ref{l:concatenation}, we obtain that $s < w$. 
We also obtain that $s < \overline{v'}$ if $d_{-\beta}(1) = \overline{v}$, $v$~primitive, hence the lexicographic characterization of~$\Omega_{-\beta}$ gives that $s \in \Omega_{-\beta}$. 
\end{proof}

The following lemma is due to Elizalde and Moore~\cite{Elizalde-Moore}, cf.\ Proposition~3 of~\cite{NguemaNdong16}. 

\begin{lemma} \label{l:bw}
Let $w \in \mathbb{N}^\infty$ with $\widehat{w} = w$.
Then we have $w \in W_{-b(w)}$. 
\end{lemma}

\begin{proof}
If $w \le u$, then $b(w) = 1$ and $w\in W_{-1}$. Now suppose $w>u$.
By Corollary~1 of~\cite{Steiner13}, we have $w \in W_{-\beta}$ for some $\beta > 1$.
We have $b(w) \ge \beta$. 
If $b(w) > \beta$, then Lemma~\ref{l:b-increasing} gives that $w \in \Omega_{-b(w)}$, hence $w = d_{-b(w)}(1)$, contradicting that $w < d_{-b(w)}(1)$ by Lemma~\ref{l:b-increasing}. 
Therefore, we have $b(w) = \beta$.
\end{proof}

\begin{proof}[Proof of Proposition~\ref{p:Omega}]
If $\beta > b(\widehat{w})$, then Lemmas~\ref{l:b-increasing} and~\ref{l:bw} give that $\widehat{w} < d_{-\beta}(1)$ and $\widehat{w} < \overline{v'}$ if $d_{-\beta}(1) = \overline{v}$, $v$~primitive, thus $w \in \Omega_{-\beta}$. 
If $1 < \beta < b(\widehat{w})$, then we have $\widehat{w} > d_{-\beta}(1)$ by Lemmas~\ref{l:b-increasing} and~\ref{l:bw}, thus $w \notin \Omega_{-\beta}$. 

Suppose in the following that $b(\widehat{w})>1$.
If $d_{-b(\widehat{w})}(1)$ is not purely periodic, then $\widehat{w} = d_{-b(\widehat{w})}(1)$ by Lemmas~\ref{l:ibeta} and \ref{l:bw}, and we have $w \in \Omega_{-b(\widehat{w})}$ if and only if $w$~does not end with~$0 \widehat{w}$.

Let now $d_{-b(\widehat{w})}(1) = \overline{v}$, with $v$~primitive. 
Then $w \in \Omega_{-b(\widehat{w})}$ implies that $\widehat{w} \le \overline{v}$.
If $|v|$ is even, then we have $\widehat{w} \ge \overline{v}$ by Lemmas~\ref{l:ibeta}, \ref{l:concatenation} and~\ref{l:bw}, thus $w \in \Omega_{-b(\widehat{w})}$ if and only if $\widehat{w} = \overline{v}$ and $w$~does not end with~$0 \overline{v}$.
Let $|v|$ be odd in the following.
Then we have $\widehat{w} \ge \overline{v'}$ by Lemmas~\ref{l:ibeta}, \ref{l:concatenation} and~\ref{l:bw}.
Note that $\overline{v'} \le \widehat{w} \le \overline{v}$ implies that $\widehat{w} = \overline{v'}$ or $\widehat{w} = \overline{v}$.
Therefore, $w \in \Omega_{-b(\widehat{w})}$ implies that $\widehat{w} \in \{\overline{v'}, \overline{v}\}$.
Recall that $w \in \Omega_{-b(\widehat{w})}$ means that $0 \overline{v'} < w_{[k,\infty)} \le \overline{v}$ for all $k \ge 1$, in particular $w$ does not end with $0 \overline{v'}$ or  $0 \overline{v}$ (as $0\overline{v} < 0\overline{v'}$).
Let now $\widehat{w} \in \{\overline{v},  \overline{v'}\}$.
If, for some $i \ge 1$, $w_{[i,\infty)}$ starts with~$v$, then $\widehat{w} \le \overline{v}$ gives that $w_{[i,\infty)} = \overline{v}$.
Therefore, $w_{[k,\infty)} \le 0\overline{v'}$ means that $w_{[k,\infty)} = 0\overline{v'}$ (hence $\widehat{w} = \overline{v'}$) or $w_{[k,\infty)} = 0 (v')^j \overline{v}$ for some $j \ge 0$ (hence $\widehat{w} = \overline{v}$). 
As $v'$ ends with~$0$, this yields that $w$ ends with~$0 \widehat{w}$ if $w \notin \Omega_{-b(\widehat{w})}$.
\end{proof}

Note that, if $d_{-b(\widehat{w})}(1) = \overline{v}$, $\widehat{w} = \overline{v'}$ and $w$ does not end with~$0\overline{v'}$, then the supremum in the definition of~$\widehat{w}$ is not attained.
For example, if $w = 110 1 (10)^2 1 (10)^3 1 \cdots$, then $\widehat{w} = \overline{10}$, thus $b(\widehat{w}) = 2$, $d_{-2}(1) = \overline{2}$, and $w \in \Omega_{-2}$.

\section{Proofs of the main results}
\begin{proof}[Proof of Theorem~\ref{t:main}]
By Lemma~\ref{l:a}, $b(a)$ is well defined. 
Suppose first $\beta > b(a)$.
Recall that $\mathrm{Pat}(x, T_{-\beta}, n) = \mathrm{Pat}(d_{-\beta}(x), \Sigma, n)$ for all $x \in (0,1]$.
By Proposition~\ref{p:pat}, we have $\mathrm{Pat}(d_{-\beta}(x), \Sigma, n) = \pi$ if and only if 
\begin{gather}
d_{-\beta,j}(x) - d_{-\beta,i}(x) \ge z_j - z_i \ \mbox{for all}\ 1 \le i, j < n\ \mbox{with}\ \pi(j) > \pi(i), \label{e:dz} \\
T_{-\beta}^{n-1}(x) > T_{-\beta}^{\ell-1}(x)\ \mbox{if}\ \pi(n) \ne 1,\ \mbox{and} \
T_{-\beta}^{n-1}(x) < T_{-\beta}^{r-1}(x)\ \mbox{if}\ \pi(n) \ne n. \label{e:Tlnr}
\end{gather}
Let $w = w_1 w_2 \cdots$ with $w_{[m,\infty)} = a$, $w_{[1,m)} = z_{[1,m)}$ if $\pi$ is not collapsed, $w_{[1,m)} = z^{(i)}_{[1,m)}$ for $i$ as in the defnition of~$a$ if $\pi$ is collapsed. 
As $\widehat{w} = a$ by Lemma~\ref{l:a} and its proof, we have $w \in \Omega_{-\beta}$ by Proposition~\ref{p:Omega}. 
Let $x \in (0,1]$ be such that $d_{-\beta}(x) = w$.  
Then \eqref{e:dz} holds. 
If $\pi(n) \ne 1$, then we have $w_{[\ell,\infty)} \le w_{[n,\infty)}$ by the proof of Lemma~\ref{l:a} and thus $T^{\ell-1}_{-\beta}(x) \le T^{n-1}_{-\beta}(x)$.
For $\pi(n) \ne n$, we have $w_{[n,\infty)} \le w_{[r,\infty)}$ and thus $T^{n-1}_{-\beta}(x) \le T^{r-1}_{-\beta}(x)$.
As $\beta>b(a)$, we know that $a \ne d_{-\beta}(1)$ by Lemmas~\ref{l:Wbeta} and~\ref{l:bw}. Then $T_{-\beta}^{k-1}(x) \ne 1$ for all $k \ge 1$, and the following results hold for all sufficiently small $\varepsilon > 0$.
We have $d_{-\beta,k}(x\pm\varepsilon) = d_{-\beta,k}(x)$ for all $1 \le k < n$, thus \eqref{e:dz} holds for $x\pm\varepsilon$ and $T^k_{-\beta}(x\pm\varepsilon) = T^k_{-\beta}(x) \pm (-\beta)^k \varepsilon$ for all $0 \le k< n$. 
If $\pi(n) \ne 1$, then we obtain that
\[
T_{-\beta}^{n-1}(x-(-1)^n \varepsilon) = T_{-\beta}^{n-1}(x) + \beta^{n-1} \varepsilon > T_{-\beta}^{\ell-1}(x) + \beta^{\ell-1} \varepsilon \ge T_{-\beta}^{\ell-1}(x-(-1)^n \varepsilon).
\]
For $\pi(n) = n$, this implies that $\mathrm{Pat}(x-(-1)^n\varepsilon, T_{-\beta}, n) = \pi$.
If $\pi(n) \ne n$, then
\[
T_{-\beta}^{n-1}(x+(-1)^n \varepsilon) = T_{-\beta}^{n-1}(x) - \beta^{n-1} \varepsilon < T_{-\beta}^{r-1}(x) - \beta^{r-1} \varepsilon \le T_{-\beta}^{r-1}(x+(-1)^n \varepsilon).
\]
This implies that $\mathrm{Pat}(x+(-1)^n\varepsilon, T_{-\beta}, n) = \pi$ in case $\pi(n) = 1$. 
If $\pi(n) \notin \{1,n\}$, then we have $\overline{w_{[\ell,n)}} < \overline{w_{[r,n)}}$, by the definition of $z^{(i)}_{[1,n)}$ if $\pi$ is collapsed, by Lemmas~\ref{l:a} and~\ref{l:period} otherwise. 
For even $n-m$, we have $w_{[n,\infty)} = \overline{w_{[\ell,n)}} < \overline{w_{[r,n)}}$ and thus $w_{[n,\infty)} < w_{[r,\infty)}$ by Remark~\ref{r:pat}, hence $T^{n-1}_{-\beta}(x\pm\varepsilon) < T^{r-1}_{-\beta}(x\pm\varepsilon)$, which implies that $\mathrm{Pat}(x-(-1)^n\varepsilon, T_{-\beta}, n) = \pi$.
For odd $n-m$, we have $w_{[\ell,\infty)} < w_{[n,\infty)} = w_{[r,\infty)}$, thus $T^{\ell-1}_{-\beta}(x\pm\varepsilon) < T^{n-1}_{-\beta}(x\pm\varepsilon)$ and $\mathrm{Pat}(x+(-1)^n\varepsilon, T_{-\beta}, n) = \pi$.
We have shown that $\pi \in \mathcal{A}(T_{-\beta})$ for all $\beta > b(a)$.

Let now $w \in \mathbb{N}^\infty$ with $\mathrm{Pat}(w, \Sigma, n) = \pi$.
We show that $\widehat{w} \ge a$, thus $\pi \notin \mathcal{A}(T_{-\beta})$ for all $\beta < b(a)$ by Proposition~\ref{p:Omega}.
If $\pi$ is not collapsed and $w_{[1,n)} \ne z_{[1,n)}$, then Propositon~\ref{p:pat} gives that $\max_{1\le k<n} w_k > \max_{1\le k<n} z_k = a_1$.
If $\pi$ is collapsed and $w_{[1,n)} \ne z^{(i)}_{[1,n)}$ for all $0 \le i < |r-\ell|$, then we have $w_m > z_m = a_1$ by Lemma~\ref{l:zi}.
For $w_{[1,n)} = z_{[1,n)}$ and $w_{[1,n)} = z^{(i)}_{[1,n)}$ respectively, we have
\[
w_{[m,\infty)} > \begin{cases}w_{[m,n)}\, \overline{w_{[\ell,n)}} & \mbox{if $n-m$ is even and $\pi(n) \ne 1$}, \\ w_{[m,n)}\, \overline{w_{[r,n)}} & \mbox{if $n-m$ is odd},\end{cases}
\]
by Propositon~\ref{p:pat} and Remark~\ref{r:pat}, thus $w_{[m,\infty)} > a$ in these cases.
If $n-m$ is even and $\pi(n) = 1$, then we cannot have $\widehat{w} < a$ because this would imply that $\overline{z_{[m,n)} 0} > \widehat{w} \ge w_{[m,\infty)} \ge z_{[m,n)} 0 \widehat{w} > \overline{z_{[m,n)} 0}$, a contradiction.
This proves that $B_-(\pi) = b(a)$.

Suppose now that $\pi \in \mathcal{A}(T_{-b(a)})$, with $b(a) > 1$, i.e., $\mathrm{Pat}(w, \Sigma, n) = \pi$ for some $w \in \Omega_{-b(a)}$.
For $n-m$ even and $\pi(n) = 1$, the previous paragraph and $w \in \Omega_{-b(a)}$ give the contradiction $d_{-b(a)}(1) \ge \widehat{w} \ge a = \overline{z_{[m,n)}0} > d_{-b(a)}(1)$; 
the last inequality is a consequence of $a \in W_{-b(a)}$ (by Lemmas~\ref{l:a} and~\ref{l:bw}) and Lemma~\ref{l:ibeta}, since $a \ne d_{-b(a)}(1)$ as $d_{-b(a)}(1)$ is not periodic with a period ending in~$0$, and $a \ne \overline{v'}$ for $d_{-b(a)}(1) = \overline{v}$, $v$~primitive, $|v|$~odd, as $z_{[m,n)}0$ and $v'$ are primitive by Lemmas~\ref{l:period} and~\ref{l:ibeta} and have different parity.
For odd $n-m$ or $\pi(n) \ne 1$, we obtain that $a < w_{[m,\infty)} \le d_{-b(a)}(1)$.
Now, we can only have $a = \overline{v'}$ with $d_{-b(a)}(1) = \overline{v}$, $|v|$~odd, $v$~primitive, and $w_{[m,\infty)} = (v')^k \overline{v}$ for some $k \ge 0$. 
As $w$ does not end with~$0 \overline{v}$ (since $w \in \Omega_{-b(a)}$), we have $k=0$.
Then $w_{[m,\infty)} = w_{[m+|v|,\infty)}$, thus $|v| > n-m$ by the definition of $\mathrm{Pat}(w, \Sigma, n)$, and $a = \overline{v'}$ (with $v'$ primitive by Lemma~\ref{l:ibeta}) implies that $n-\ell\ge|v'|>n-m$, hence $\ell < m$ if $n-m$ is even, $n-r\ge|v'| >n-m$ $r < m$ if $n-m$ is odd.
As $v'$ ends with~$0$, we have $z_{m-1} = 0$ if $\pi$ is not collapsed, $z^{(i)}_{m-1} = 0$ if $\pi$ is collapsed, where $i$ gives the minimum in the definition of~$a$.
However, $w \in \Omega_{-b(a)}$ implies that $w_{m-1} > 0$.  
Then, by Proposition~\ref{p:pat}, we have $\max_{1\le k<n} w_k > a_1$ or $w_m = z_m+1$, $\pi$ collapsed. 
In the latter case, $w_{[1,n)} = z^{(j)}_{[1,n)}$ for some $j \ne i$ by Lemma~\ref{l:zi}.
If $\max_{1\le k<n} w_k > a_1$, then $\widehat{w} \le d_{-b(a)}(1)$ implies that $v' = a_10$, $v = (a_1{+}1)$, hence $m=n$ (as $|v| > n-m$) and $w_{[\ell,\infty)} = \overline{(a_1{+}1)}$ (as $\max_{1\le k<n} w_k = w_\ell$ and $\widehat{w} \le d_{-b(a)}(1)$); 
this gives $w_{[\ell,\infty)} =w_{[m,\infty)}$, contradicting that $\mathrm{Pat}(w, \Sigma, n) = \pi$.
If $\pi$ is collapsed and $w_{[1,n)} = z^{(j)}_{[1,n)}$ for some $j \ne i$, then we have $\overline{v} = w_{[m,\infty)} > \overline{z^{(j)}_{[m,n)} z^{(j)}_{[h,m)}} > \overline{z^{(i)}_{[m,n)} z^{(i)}_{[h,m)}} = \overline{v'}$, with $h = \ell$ if $n-m$ is even or $h = r$ if $n-m$ is odd, which implies that $\overline{z^{(j)}_{[m,n)} z^{(j)}_{[h,m)}} \in W_{-b(a)}$ by Lemmas~\ref{l:ibeta} and~\ref{l:concatenation} and the proof of Lemma~\ref{l:a}, contradicting that no sequence in $W_{-b(a)}$ is strictly between $\overline{v'}$ and~$\overline{v}$.
Therefore, we have $\pi \notin \mathcal{A}(T_{-b(a)})$ for $b(a) > 1$.

Finally, for $b(a) > 1$, the eventual periodicity of $a$ implies that $d_{-b(a)}(1)$ is eventually periodic by Lemma~\ref{l:ibeta}, i.e., $b(a)$ is an Yrrap number.
\end{proof}

\begin{proof}[Proof of Theorem~\ref{t:1}]
By Theorem~\ref{t:main}, we have $B_-(\pi) = b(a)$ (with $\widehat{a} = a$ by Lemma~\ref{l:a}).
If $b(a) = 1$, then Lemmas~\ref{l:Wbeta} and~\ref{l:bw} and imply that $a = \overline{\varphi^k(0)}$ for some $k \ge 0$, as $a = u$ is not possible since $a$ is eventually periodic and $u$ is aperiodic. 
If $a = \overline{\varphi^k(0)}$ for some $k \ge 0$, then we have $b(a) = 1$ by
the proof of Lemma~\ref{l:Wbeta}. 
\end{proof}

\begin{proof}[Proof of Theorem~\ref{t:N}]
We first prove that $\lfloor B_-(\pi)\rfloor = c + \epsilon$, with $c = \max_{1\le j<n} z_j$.
By Theorem~\ref{t:main}, we have $B_-(\pi) = b(a)$. 
If $\pi$ is not collapsed, then we have $\lfloor b(a)\rfloor = c$ if $a \ne \overline{c0}$, $b(a) = c+1$ if $a = \overline{c0}$.
If $\pi$ is collapsed, then we have $\lfloor b(a)\rfloor = c+1$.
Indeed, $a = \overline{(c{+}1)0}$ is impossible because this would imply that $\overline{z_{[\ell,n)}}$ or $\overline{z_{[r,n)}}$ is equal to $\overline{c0}$ or~$\overline{0c}$, hence $c = 0$ and $|r-\ell| = 1$ by Lemma~\ref{l:period}, i.e., $z^{(0)}_{[\ell,n)} = 01$ if $\ell < r$, $z^{(0)}_{[r,n)} = 10$ if $r < \ell$; if $n-m$ is even, then $z^{(0)}_{[m,n)} = (10)^{(n-m)/2}$ implies that $r < \ell$, hence $\overline{z^{(0)}_{[\ell,n)}} = \overline{0}$, if $n-m$ is odd, then we have $\ell < r$, $\overline{z^{(0)}_{[r,n)}} = \overline{1}$.

As $\pi \in \mathcal{A}(T_{-\beta})$ for all $\beta > b(a)$ by Theorem~\ref{t:main}, we obtain that $N_-(\pi) \le \lfloor b(a)\rfloor +1$.
We have $N_-(\pi) \ge c+1$ by Propositon~\ref{p:pat}, and $N_-(\pi) \ge c+2=c+1+\epsilon$ if $\pi$ is collapsed by Lemma~\ref{l:zi}.
It remains to prove that $N_-(\pi) \ge c+2$ when $a = \overline{c0}$. 
This holds for $a = \overline{0}$ since $n \ge 2$.
If $a = \overline{c0}$ with $c \ge 1$, then we cannot have $n-m$ even and $\pi(n) = 1$. 
Therefore, the proof of Theorem~\ref{t:main} gives that $w_{[m,\infty)} > a$ for all $w \in \mathbb{N}^\infty$ satisfying $\mathrm{Pat}(w, \Sigma, n) = \pi$, hence $\max_{k\ge1} w_k > c$ and thus $N_-(\pi) \ge c+2$.

In the following four cases, we have $N_-(\pi) = n-1$.
\begin{itemize}
\item
If $\pi = 12\cdots n$, then $z_{[1,n)} = 01\cdots (n{-}2)$ and $a = \overline{n{-}2}$. 
\item
If $\pi = 12\cdots (n{-}2)n(n{-}1)$, then $z_{[1,n)} = 01\cdots(n{-}3)(n{-}3)$, $\pi$ is collapsed.
\item
If $\pi = n(n{-}1)\cdots1$, then $z_{[1,n)} = (n{-}2)\cdots 10$ , $a = \overline{z_{[1,n)}0}$ or 
$a = z_{[1,n)} \overline{0}$. 
\item
If $\pi = n(n{-}1)\cdots312$, then $z_{[1,n)} = (n{-}3)\cdots 100$ and $\pi$ is collapsed. 
\end{itemize}

We have $c\le n-2$, and the only permutations $\pi \in \mathcal{S}_n$ with $c = n-2$ are $12\cdots n$ and $n(n{-}1)\cdots 1$, for which $\epsilon=0$ if $n \ge 3$, thus $N_-(\pi) < n$ for all $\pi \in \mathcal{S}_n$, $n \ge 3$.

Now suppose that $n\ge 4$ and $N_-(\pi) = n-1$. 
If $\epsilon=0$, then we have $c = n-2$, thus $\pi = 12\cdots n$ or $\pi = n(n{-}1)\cdots 1$.
If $\epsilon=1$, i.e., if $\pi$ is collapsed or $a = \overline{c0}$, then we have $c=n-3$.
If $\pi$ is collapsed, then $c=n-3$ implies that $|r-\ell| = 1$ (since $z_{[1,n)}$ contains all digits $0,1,\ldots,c$), and $\pi = 12\cdots (n{-}2)n(n{-}1)$ or $\pi = n(n{-}1)\cdots312$.
If $a = \overline{c0}$ with $c = n-3$ and $\pi$ not collapsed, then $m \ge n-3$ since $z_{[m,n)}$ is a prefix of~$a$. 
For $m=n-3$, we have $\pi(n-3) > \pi(n-1)$, and $a = \overline{c0}$ with $c \ge 1$ implies that $r = n-2$, 
thus $\pi(n-2) > \pi(n)$; this gives that $z_{n-3} > z_{n-1}$, contradicting that  $a = \overline{c0}$.
For $m=n-2$, $a = \overline{c0}$ with $c \ge 1$ is not possible. 
For $m=n-1$, $a = \overline{c0}$ with $c \ge 1$ implies that $r = n-2$ and $z_r = 0$; if $\pi(n) \in \{1,2\}$, then $\tilde{\pi}$ with $\pi(1)$ removed has at most $n-4$ ascents; if $\pi(n) \ge 3$, then $z_i = 0$ for at least $3$ indices~$i$, thus $c \le n-4$. 
By similar arguments, we cannot have $m=n$ and $a = \overline{(n{-}3)0}$. 

For $n = 3$, Table~\ref{ta:1} gives that $\max_{\pi \in \mathcal{S}_n}B_-(\pi) = b(1 \overline{0}) = \frac{1+\sqrt{5}}{2}$, and the maximum is attained only for $\pi = 312$.
In the following, let $n \ge 4$. 
We have just seen that $B_-(\pi) < n-2$ for all but $4$ permutations $\pi \in \mathcal{S}_n$. 
Moreover, we have 
\begin{gather*}
B_-(12\cdots n) = B_-(12\cdots (n{-}2)n(n{-}1)) = b(\overline{n{-}2}) = n-2, \\
B_-(n(n{-}1)\cdots1) = 	\begin{cases}
							b( \overline{(n{-}2) (n{-}3) \cdots 1 0 0}) = b( \overline{(n{-}2) (n{-}3) \cdots 1 1})	& \text{if } n  \text{ is odd},  \\
							b((n{-}2) (n{-}3) \cdots 1 \overline{0})		& \text{if } n  \text{ is even}, 
							\end{cases} \\
B_-(n(n{-}1)\cdots312) =  	\begin{cases}
								b((n{-}2) (n{-}3) \cdots 1 \overline{0}) 	& \text{if } n  \text{ is odd},  \\
								b((n{-}2) (n{-}3) \cdots 2 \overline{10})& \text{if } n  \text{ is even}.
								\end{cases}
\end{gather*}
By Theorem~2 of~\cite{Steiner13}, we have $d_{-b(w)}(1) = w$ for all 
\[
w \in \big\{(n{-}2) (n{-}3) \cdots 1 \overline{0},\, \overline{n-2},\, \overline{(n{-}2) (n{-}3) \cdots 1 1},\, (n{-}2) (n{-}3) \cdots 2 \overline{10}\big\},
\]
except for $w = 2\overline{10}$, as $b(2\overline{10}) = 2$. 
For $w = (n{-}2) (n{-}3) \cdots 1 \overline{0}$, we have $\lfloor b(w) \rfloor = n-2$, $w > \overline{n-2}$, $w > (n{-}2) (n{-}3) \cdots 2 \overline{10}$ if $n$ is even, $w > \overline{(n{-}2) (n{-}3) \cdots 1 1}$ if $n$ is odd.
Hence $\max_{\pi \in \mathcal{S}_n}B_-(\pi) = b(w) \in (n-2,n-1)$, and the maximum is attained only for $\pi = n(n{-}1)\cdots1$ if $n$ is even, $\pi = n(n{-}1)\cdots 312$ if $n$ is odd. 
\end{proof}

\begin{proof}[Proof of Theorem~\ref{t:Yrrap}]
By Theorem~\ref{t:main}, $B_-(\pi)$ is an Yrrap number for all $\pi \in \mathcal{S}_n$ with $B_-(\pi) > 1$.
Let now $\beta > 1$ be an Yrrap number, $w = d_{-\beta}(1)$, and $p,q \ge 1$ minimal such that $w_{[p+q,\infty)} = w_{[q,\infty)}$, i.e., $d_{-\beta}(1) = w_{[1,q)} \overline{w_{[q,p+q)}}$. 
Define $\varrho \in \mathcal{S}_{p+q}$ by
\begin{itemize}
\itemsep.5ex
\item
$\varrho(i) < \varrho(j)$ if $w_{[i,\infty)} < w_{[j,\infty)}$, $1 \le i, j < p+q$, 
\item
$\varrho(p+q) = \varrho(q) - (-1)^{p+q}$.
\end{itemize}
We now define $\pi$ by increasing the differences in~$\varrho$ and putting the missing elements at the beginning, ordered by growth. 
More precisely, define integers~$y_j$ recursively for $j = \varrho^{-1}(p+q), \ldots, \varrho^{-1}(3), \varrho^{-1}(2)$, $1 \le i \le p+q$ such that $\varrho(i) = \varrho(j)-1$, by
\[
y_j = \begin{cases}0 & \mbox{if}\ w_j = w_i,\ \mbox{or}\ w_j = w_i+1,\ \varrho(i+1) < \varrho(j+1), \\
1 & \mbox{if}\ w_j = w_i+1,\, \varrho(i+1) > \varrho(j+1),\ \varrho(i+1) < \varrho(j)\ \mbox{or} \\ 
& \quad y_k \ge 1\ \mbox{for some}\ 1 \le k \le p+q\ \mbox{with}\ \varrho(j) < \varrho(k) \le \varrho(j+1), \\
2 & \mbox{if}\ w_j = w_i+1,\ \varrho(i+1) > \varrho(j+1),\ \varrho(i+1)  > \varrho(j)\ \mbox{and} \\
& \quad y_k = 0\ \mbox{for all}\ \varrho(j) < \varrho(k) \le \varrho(j+1). \\
1 & \mbox{if}\ w_j = w_i+2,\ \varrho(i+1) < \varrho(j)\ \mbox{and}  \\
& \quad y_k \ge 1\ \mbox{for some}\ \varrho(j) < \varrho(k) \le \varrho(j+1), \\
w_j- w_i & \mbox{if}\ w_j \ge w_i+2,\ \varrho(i+1) < \varrho(j)\ \mbox{or}\ \\
& \quad y_k \ge 1\ \mbox{for some}\ \varrho(j) < \varrho(k) \le \varrho(j+1), \\
w_j - w_i + 1 & \mbox{if}\ w_j \ge w_i+2,\ \varrho(i+1) > \varrho(j), \ \mbox{and} \\
& \quad y_k \ge 1\ \mbox{for some}\ \varrho(j) < \varrho(k) \le \varrho(j+1), \\
w_j - w_i - 1 & \mbox{if}\ w_j \ge w_i+3,\ \varrho(i+1) < \varrho(j)\ \mbox{and} \\
& \quad y_k = 0\ \mbox{for all}\ \varrho(j) < \varrho(k) \le \varrho(j+1),
\end{cases}
\]
with $\varrho(p+q+1) = \varrho(q+1)$. 
For $\varrho(j) = 1$, set
\[
y_j = w_j + \begin{cases}1 & \mbox{if}\ y_k = 0\ \mbox{for all}\ 1 \le k \le p+q\ \mbox{with}\ 1 < \varrho(k) \le \varrho(j+1), \\ 0 & \mbox{otherwise.}\end{cases}
\]
Set $c = \sum_{j=1}^{p+q} y_j$, $n = c+p+q$, and define $\pi \in \mathcal{S}_n$ by 
\begin{itemize}
\itemsep.5ex
\item
$\pi(c+j) = \varrho(j) + \sum_{1\le k\le p+q:\, \varrho(k)\le\varrho(j)} y_k$ for $1 \le j \le p+q$, 
\item
$\pi(i) < \pi(j)$ if $1 \le i < j \le c$. 
\end{itemize}

With the notation of Theorem~\ref{t:main}, we show that $a = w$ for this choice of~$\pi$, thus $B_-(\pi) = b(d_{-\beta}(1))=\beta$.
First note that $m = c + \varrho^{-1}(p+q)$ and $\pi(n) = \pi(c+q) - (-1)^{p+q}$ since $w_{p+q} = w_q$, thus 
\[
\begin{cases}
m = c+1,\ r = c+q, & \mbox{if}\ q \ge 1,\ \mbox{$p+q$ even}, \\
m = c+1,\ \ell = c+q, & \mbox{if}\ q \ge 2,\ \mbox{$p+q$ odd}, \\
m = n,\ \ell = c+1, & \mbox{if}\ q = 1,\ \mbox{$p+q$ odd}.
\end{cases}
\]
Therefore, $z_{[m,n)} \overline{z_{[\ell,n)}}$ and $z_{[m,n)} \overline{z_{[r,n)}}$ respectively are equal to $z_{[c+1,c+q)} \overline{z_{[c+q,n)}}$.
To prove that $z_{c+i} = w_i$ for all $1 \le i \le p+q$, with $z_n = z_{c+q}$, compare $z_{c+j} - z_{c+i}$ to $w_j - w_i$ for $\varrho(i) = \varrho(j)-1 \ne 0$.
Let $h = \sum_{1\le k\le p+q:\, \varrho(k)\le\varrho(j)} y_k$.
Then we have 
\[
\pi(h-k) = \pi(h)-k\ \mbox{for}\ 0 \le k < y_j,\ \pi(h-y_j+1) = \pi(c+i)+1,\ \pi(h) = \pi(c+j)-1.
\]
We obtain that $z_{h-k} = z_{h-k-1} + 1$ for $1 \le k \le y_j-2$.
For $y_j \ge 1$, we have 
\[
z_{c+j} = z_h + \begin{cases}1 & \mbox{if}\ y_k \ge 1\ \mbox{for some}\ 1 \le k \le p+q\ \mbox{with}\ \varrho(j) < \varrho(k) \le \varrho(j+1), \\
0 & \mbox{otherwise}.\end{cases}
\]
Indeed, for $j \ne p+q$, we have $\pi(c+j+1) > \pi(h+1)$ if and only if $y_k \ge 1$ for some $1 \le k \le p+q$ with $\varrho(j) < \varrho(k) \le \varrho(j+1)$.
For $j = p+q$ odd, we have $i = q$, thus $w_j = w_i$ and $y_j = 0$.
For $j = p+q$ even, we have $z_{c+j} = z_{c+q}$, and $y_k \ge 1$ for some $\varrho(p+q) < \varrho(k) \le \varrho(p+q+1)$ is equivalent to $y_k \ge 1$ for some $\varrho(q) < \varrho(k) \le \varrho(q+1)$ (as $\varrho(p+q+1) = \varrho(q+1)$, $\varrho(q) = \varrho(p+q) + 1$ and $y_q = 0$). 
For $y_j \ge 2$, we have 
\[
z_{h-y_j+1} = z_{c+i} + \begin{cases}1 & \mbox{if}\ \varrho(i+1) < \varrho(j), \\
0 & \mbox{otherwise.} \end{cases}
\]
Here, $\pi(h-y_j+2) > \pi(c+i+1)$ is equivalent to $\varrho(i+1) < \varrho(j)$ for $i \ne p+q$; for $i = p+q$, we have $z_{c+i} = z_{c+q}$ and $\varrho(p+q+1) = \varrho(q+1)$.
If $y_j = 1$, then 
\[
z_h = z_{c+i} + \begin{cases}1 & \mbox{if}\ y_k = 0\ \mbox{for all}\ 1 \le k \le p+q\ \mbox{with}\ \varrho(j) < \varrho(k) \le \varrho(i+1), \\
0 & \mbox{otherwise.} \end{cases}
\]
Finally, if $y_j = 0$, then $z_{c+j} = z_{c+i} + 1$ if $\varrho(i+1) > \varrho(j+1)$, $z_{c+j} = z_{c+i}$ otherwise. 
Summing up these differences shows that $z_{c+j} - z_{c+i} = w_j - w_i$ for $\varrho(i) = \varrho(j)-1 \ne 0$.
For $\varrho(j) = 1$, we have $z_k = k-1$ for all $k \le h = y_j$ and $z_{c+j} - z_h$ is given above, thus $z_{c+j} = w_j$. 

Therefore, we have $z_{[c+1,n)} = w_{[1,n)}$. 
The minimality of $p$ and~$q$ gives that $\pi$ is not collapsed, hence $a = w_{[1,q)} \overline{w_{[q,p+q)}} = d_{-\beta}(1)$. 
\end{proof}

\bibliographystyle{ws-ijfcs}
\bibliography{permutations}
\label{sec:biblio}
\end{document}